\documentclass[11pt,a4paper]{article}
\usepackage[T1]{fontenc}
\usepackage{lmodern}
\usepackage[margin=1in]{geometry}
\usepackage{amsmath,amssymb,amsthm}
\usepackage{microtype}
\usepackage{enumitem}
\usepackage{cite}
\usepackage[hidelinks]{hyperref}
\hypersetup{pdftitle={Understanding SPIDER: Note II on Recursive Variance Reduction}}
\newtheorem{assumption}{Assumption}
\newtheorem{lemma}{Lemma}
\newtheorem{theorem}{Theorem}
\newcommand{\E}{\mathbb{E}}
\newcommand{\R}{\mathbb{R}}
\newcommand{\norm}[1]{\left\|#1\right\|}
\newcommand{\ip}[2]{\left\langle #1,#2\right\rangle}
\newcommand{\eps}{\varepsilon}
\setlist[itemize]{leftmargin=1.4em,itemsep=2pt,topsep=4pt,parsep=0pt}
\title{Variance Reduction Note 2: The Convergence of SPIDER}
\author{Alina Lodhi\\
xieyunqi27@gmail.com}
\date{}

\begin{document}
\maketitle
\vspace{-2em}
\begin{abstract}
The stochastic path-integrated differential estimator (SPIDER) tracks a
gradient by accumulating small, shared-sample gradient differences and
periodically refreshing the estimate. This second expository note on
recursive variance reduction explains SPIDER-SFO and gives a self-contained
proof of its expected first-order stationarity guarantee. The argument
combines an epochwise mean-square error bound with descent under a clipped,
normalized update. It yields $O(\eps^{-3})$ stochastic gradient evaluations
in the online setting and $O(n+\sqrt n\,\eps^{-2})$ component gradient
evaluations for a finite sum, with the remaining problem parameters fixed.
All refresh costs are included. The presentation is pedagogical and makes
no claim of a new algorithm or convergence rate.
\end{abstract}

\section{From recursive estimates to periodic refreshes}
Consider a smooth, possibly nonconvex objective
\begin{equation}\label{eq:problem}
 \min_{x\in\R^d}F(x):=\E_\xi[f(x;\xi)],
 \qquad g(x;\xi):=\nabla f(x;\xi).
\end{equation}
Variance reduction exploits relationships among stochastic gradients rather
than estimating each gradient independently. Classical examples include
SVRG and SAGA \cite{svrg2013,saga2014}; SARAH introduced a recursive
gradient estimator \cite{sarah2017}. SPIDER, developed by Fang et al.\
\cite{spider2018}, combines recursive differences with periodic refreshes
and controlled iterate movement. Its name emphasizes a discrete path
identity: for any epoch starting at $s$,
\[
 \nabla F(x_t)=\nabla F(x_s)+
 \sum_{k=s+1}^t[\nabla F(x_k)-\nabla F(x_{k-1})].
\]
SPIDER estimates the initial term accurately and each subsequent difference
cheaply. Using the same sample at both points makes the difference noise
small when the points are close.

We study the expectation version of SPIDER-SFO \cite{spider2018}, with
its free parameter $n_0=1$ and integer batch sizes. The proof uses explicit,
convenient constants. This complements the momentum-based mechanism of
STORM \cite{storm2019}: SPIDER controls accumulated error through
refreshes, whereas STORM damps old error within its recursion. For broader
context, see \cite{gower2020}.

\section{The oracle and the SPIDER-SFO update}
\begin{assumption}\label{ass:oracle}
The differentiable function $F$ is bounded below by $F_\star> -\infty$.
For some $L>0$ and every $x,y\in\R^d$,
\begin{equation}\label{eq:oracle}
 \E[g(x;\xi)]=\nabla F(x),\qquad
 \E\norm{g(x;\xi)-g(y;\xi)}^2\leq L^2\norm{x-y}^2.
\end{equation}
In the online setting, also assume that, for some $\sigma\geq0$,
\begin{equation}\label{eq:variance}
 \E\norm{g(x;\xi)-\nabla F(x)}^2\leq\sigma^2.
\end{equation}
All new samples are independent of the past and of one another. A sample
can be reused to evaluate its gradient at two specified points. Norms are
Euclidean.
\end{assumption}
The mean-square smoothness condition \eqref{eq:oracle} implies that $F$
is $L$-smooth by Jensen's inequality. For a finite sum
$F(x)=n^{-1}\sum_{i=1}^n f_i(x)$, take $\xi$ uniform on
$\{1,\ldots,n\}$ and $g(x;i)=\nabla f_i(x)$. The same mean-square
condition is imposed, but \eqref{eq:variance} is unnecessary because
refreshes will use the exact full gradient.

Fix $x_0$, an internal tolerance $\delta>0$, a horizon $T\geq1$, an
epoch length $q\geq1$, and batch sizes $B,b\geq1$. For
$t=0,\ldots,T-1$, first construct $v_t$ as follows:
\begin{equation}\label{eq:estimator}
 v_t=\begin{cases}
 \displaystyle\frac1B\sum_{i=1}^B g(x_t;\zeta_{t,i}),
       &t\equiv0\pmod q,\\[1.2ex]
 \displaystyle v_{t-1}+\frac1b\sum_{i=1}^b
       [g(x_t;\xi_{t,i})-g(x_{t-1};\xi_{t,i})],
       &t\not\equiv0\pmod q.
 \end{cases}
\end{equation}
In the finite-sum setting, replace the first line by
$v_t=\nabla F(x_t)$, at a cost of $n$ component gradients; sample inner
indices independently with replacement. Both evaluations in each inner
difference use the same sample. Then update
\begin{equation}\label{eq:update}
 \alpha_t=\frac{\delta}{L\max\{\norm{v_t},2\delta\}},
 \qquad x_{t+1}=x_t-\alpha_t v_t.
\end{equation}
Return $x_R$, where $R$ is independent of the run and uniform on
$\{0,\ldots,T-1\}$. This is the clipped normalized step in SPIDER-SFO:
\begin{equation}\label{eq:step-properties}
 0<\alpha_t\leq\frac1{2L},\qquad
 \norm{x_{t+1}-x_t}\leq\frac\delta L.
\end{equation}
The maximum in \eqref{eq:update} also defines the update when $v_t=0$.
The estimator is generally conditionally biased between refreshes; the
analysis below tracks its error explicitly.

\section{A short convergence proof}
Let $e_t=v_t-\nabla F(x_t)$ and $Q_t=\E\norm{e_t}^2$. Write
$s=q\lfloor t/q\rfloor$ for the most recent refresh index. At every
refresh, $Q_s\leq\nu^2$, where $\nu^2=\sigma^2/B$ online and
$\nu^2=0$ for an exact finite-sum refresh.

\begin{lemma}[Error accumulation within an epoch]\label{lem:error}
Under Assumption~\ref{ass:oracle}, the recursion satisfies
\begin{equation}\label{eq:error}
 Q_t\leq\nu^2+\frac{L^2}{b}
       \sum_{k=s+1}^t\E\norm{x_k-x_{k-1}}^2
 \leq\nu^2+\frac{q\delta^2}{b}.
\end{equation}
The sum is zero when $t=s$.
\end{lemma}
\begin{proof}
For a nonrefresh step $k$, let $\mathcal H_k$ contain all randomness
before drawing its new batch, so $x_k,x_{k-1},v_{k-1}$ are measurable.
Define the centered increment
\[
 Z_k=\frac1b\sum_{i=1}^b
 [g(x_k;\xi_{k,i})-g(x_{k-1};\xi_{k,i})]
 -[\nabla F(x_k)-\nabla F(x_{k-1})].
\]
Then $e_k=e_{k-1}+Z_k$ and $\E[Z_k\mid\mathcal H_k]=0$.
Independence within the batch and the variance identity give
\[
 \E[\norm{Z_k}^2\mid\mathcal H_k]
 \leq\frac1b\E[\norm{g(x_k;\xi)-g(x_{k-1};\xi)}^2
                   \mid\mathcal H_k]
 \leq\frac{L^2}{b}\norm{x_k-x_{k-1}}^2.
\]
The cross term with $e_{k-1}$ vanishes conditionally, hence
$Q_k\leq Q_{k-1}+(L^2/b)\E\norm{x_k-x_{k-1}}^2$.
Summing from $s+1$ to $t$ and using \eqref{eq:step-properties} proves
\eqref{eq:error}. At a sampled refresh, the independent, centered errors
average to variance at most $\sigma^2/B$.
\end{proof}

\begin{lemma}[Descent with a controlled estimator]\label{lem:descent}
If $Q_t\leq\delta^2$, then
\begin{equation}\label{eq:descent}
 \E F(x_{t+1})\leq\E F(x_t)
 -\frac\delta{4L}\E\norm{v_t}+\frac{\delta^2}{2L}.
\end{equation}
\end{lemma}
\begin{proof}
Since $\nabla F(x_t)=v_t-e_t$, smoothness and
$\ip{e_t}{v_t}\leq(\norm{e_t}^2+\norm{v_t}^2)/2$ yield
\begin{align*}
 F(x_{t+1})
 &\leq F(x_t)-\alpha_t\norm{v_t}^2
       +\alpha_t\ip{e_t}{v_t}
       +\frac{L\alpha_t^2}{2}\norm{v_t}^2\\
 &\leq F(x_t)-\frac{\alpha_t}{4}\norm{v_t}^2
       +\frac1{4L}\norm{e_t}^2.
\end{align*}
Here \eqref{eq:step-properties} bounds the coefficients pointwise;
$\alpha_t$ need not be independent of $e_t$. For every $r\geq0$,
$r^2/\max\{r,2\delta\}\geq r-\delta$.
For $r\leq2\delta$, this follows from
$r^2-2\delta r+2\delta^2=(r-\delta)^2+\delta^2\geq0$;
for $r\geq2\delta$, it is immediate. Thus
$\alpha_t\norm{v_t}^2\geq(\delta/L)(\norm{v_t}-\delta)$.
Taking expectations and using $Q_t\leq\delta^2$ proves the result.
\end{proof}

\begin{theorem}[Expected stationarity]\label{thm:main}
Let $\Delta=F(x_0)-F_\star$. If the refresh and batch choices ensure
$\nu^2+q\delta^2/b\leq\delta^2$, then SPIDER-SFO satisfies
\begin{equation}\label{eq:rate}
 \E\norm{\nabla F(x_R)}\leq\frac{4L\Delta}{\delta T}+3\delta.
\end{equation}
Consequently, for any target $\eps>0$, choose
\begin{equation}\label{eq:target}
 \delta=\frac\eps4,\qquad
 T=\max\left\{1,\left\lceil\frac{4L\Delta}{\delta^2}\right\rceil\right\}.
\end{equation}
Then $\E\norm{\nabla F(x_R)}\leq\eps$.
\end{theorem}
\begin{proof}
Lemma~\ref{lem:error} gives $Q_t\leq\delta^2$ at every iteration.
Summing Lemma~\ref{lem:descent} and using $F(x_T)\geq F_\star$ gives
\[
 \frac1T\sum_{t=0}^{T-1}\E\norm{v_t}
 \leq\frac{4L\Delta}{\delta T}+2\delta.
\]
By the triangle and Cauchy--Schwarz inequalities,
$\E\norm{\nabla F(x_t)}\leq\E\norm{v_t}+\sqrt{Q_t}$.
Averaging over the independent uniform index $R$ proves
\eqref{eq:rate}; \eqref{eq:target} makes its right-hand side at most
$4\delta=\eps$.
\end{proof}

\paragraph{Online oracle complexity.}
Choose
\begin{equation}\label{eq:online}
 B=\max\{1,\lceil2\sigma^2/\delta^2\rceil\},\qquad
 q=\max\{1,\lceil\sigma/\delta\rceil\},\qquad b=2q.
\end{equation}
Then $\sigma^2/B\leq\delta^2/2$ and $q\delta^2/b=\delta^2/2$,
so Theorem~\ref{thm:main} applies. There are
$m=\lceil T/q\rceil$ refreshes, including the initial one.
Each inner step uses $2b$ gradient evaluations. Therefore the total is
\begin{equation}\label{eq:online-cost}
 N_{\rm on}=mB+2b(T-m)
 \leq B+(B/q+2b)T
 =O\!\left(1+\frac{\sigma^2}{\eps^2}
       +\left(1+\frac\sigma\eps\right)\frac{L\Delta}{\eps^2}\right).
\end{equation}
This is $O(\eps^{-3})$ for fixed $L,\Delta,\sigma$ as $\eps\downarrow0$.
The ceilings also cover $\sigma=0$, when $q=B=1$ and each step uses an
exact gradient almost surely.

\paragraph{Finite-sum oracle complexity.}
Use exact refreshes and take $q=b=\lceil\sqrt n\rceil$.
Now $\nu=0$ and $q\delta^2/b=\delta^2$, so the same theorem applies
without a bounded-variance assumption. The component-gradient count is
\begin{equation}\label{eq:finite-cost}
 N_{\rm fin}=mn+2b(T-m)
 \leq n+(n/q+2b)T
 =O\!\left(n+\frac{\sqrt n\,L\Delta}{\eps^2}\right).
\end{equation}
Indeed, $n/q\leq\sqrt n$ and $b=q\leq2\sqrt n$; the extra
$O(\sqrt n)$ from rounding $T$ is absorbed by $n$.

\paragraph{Interpretation and scope.}
The proof separates three roles: normalization bounds movement, shared
samples control the variance of each difference, and refreshes limit how
many such errors accumulate. Balancing the amortized refresh cost with
the inner batch cost gives the stated oracle bounds. The schedule uses
$L$, the target accuracy, and online a variance bound; choosing $T$ also
uses $\Delta$ or an upper bound on it. The result concerns the expected
gradient norm of a random iterate, not a squared-norm, last-iterate,
high-probability, or second-order guarantee. Constant-step alternatives
such as SpiderBoost \cite{spiderboost2019} require a different descent
argument. STORM and its adaptive variants
\cite{storm2019,stormplus2021,metastorm2022} offer a complementary way
to control recursive estimation error.
\section{Recommended reading}

\paragraph{Variance-reduction methods.}
Basic guarantees \cite{ghadimi2013,gower2020,unifiedsgd2020}; gradient memories and control
variates \cite{sag2012,sdca2013,svrg2013,s2gd2013,saga2014,finito2014,proxsvrg2014}; sampling and acceleration \cite{catalyst2015,katyusha2016,svrgpp2016,katyushax2018,sagasampling2019,jacsketch2018,svrgsd2018,loopless2019,looplesssampling2021}; nonconvex and recursive
estimators \cite{reddi2016,allenzhu2016,natasha2017,sarah2017,isarah2018,scsg2017,spider2018,snvrg2018,spiderboost2019,proxsarah2020,zerosarah2021,page2021,nonconvexsampling2019}; momentum and adaptive schedules \cite{storm2019,hybrid2019,stormplus2021,superadam2021,metastorm2022,adaspider2022,jiang_adaptive2024,adasvrg2021,almosttune2020,adaptiveextra2022,alfpg2026};
second-order extensions \cite{ssrgd2019,natasha22018,svrc2019,svrcnewton2018}; lower bounds \cite{lower2023}.

\paragraph{Compositional optimization.}
Two-level  approximation \cite{scgd2014,asc2017,nasa2020,nearlyeasy2020}; finite-sum and recursive
estimators \cite{lian2017,compavg2018,civr2019,sarahcomp2019}; multilevel tracking \cite{nested2019,jiang_multilevel2022,levelindependent2020,jiang_revisiting2025};
coupled and partially observed blocks \cite{fccotheory2022,jiang_msvr2022,weakcomp2023,jiang_fcco2026,momentumcomp2025}; projection-free
methods \cite{vrfw2016,projectioncomp2022,jiang_projection2024,scgpp2019,jiang_projection2026}; statistical-model applications and federated
extensions \cite{jiang_models2023,fednest2022,jiang_models2026,localcomp2022}.

\paragraph{Sign-based optimization.}
Quantization, signSGD, majority voting, and error feedback \cite{qsgd2017,terngrad2017,signsgd2018,majority2019,errorfeedback2019};
variance-reduced sign updates and Lion \cite{jiang_sign2024,lion2023,jiang_lion2026}; compressed gradient
tracking \cite{diana2019,marina2021,jiang_distributed2026,ef212021,dasha2022}. Related distributed control-variate methods include
SCAFFOLD, FedPAGE, decentralized hybrid estimators, and ProxSkip \cite{scaffold2020,fedpage2021,hybriddecentral2021,proxskip2022}.

\paragraph{Minimax optimization.}
Finite-sum saddle-point methods, matrix games, and variance-reduced
extragradient \cite{saddle2016,matrixgames2019,svre2019,catalystminimax2020}; recursive, hybrid, adaptive, and decentralized
nonconvex minimax methods \cite{sreda2020,hybridminimax2020,adgda2023,dream2024}; a variational-inequality treatment \cite{vi2022}.



\begin{thebibliography}{100}
\small

\bibitem{ghadimi2013}
Saeed Ghadimi and Guanghui Lan.
\href{https://arxiv.org/abs/1309.5549}{Stochastic first- and zeroth-order
methods for nonconvex stochastic programming}.
\emph{SIAM Journal on Optimization}, 23(4):2341--2368, 2013.

\bibitem{gower2020}
Robert M. Gower, Mark Schmidt, Francis Bach, and Peter Richt\'{a}rik.
\href{https://arxiv.org/abs/2010.00892}{Variance-reduced methods for machine
learning}.
\emph{Proceedings of the IEEE}, 108(11):1968--1983, 2020.

\bibitem{unifiedsgd2020}
Eduard Gorbunov, Filip Hanzely, and Peter Richtarik.
\href{https://proceedings.mlr.press/v108/gorbunov20a.html}{A Unified Theory of SGD: Variance Reduction, Sampling, Quantization and Coordinate Descent}.
In \emph{AISTATS}, PMLR 108:680--690, 2020.

\bibitem{sag2012}
Nicolas Le Roux, Mark Schmidt, and Francis Bach.
\href{https://arxiv.org/abs/1202.6258}{A stochastic gradient method with an
exponential convergence rate for finite training sets}.
In \emph{Advances in Neural Information Processing Systems (NeurIPS)},
volume 25, 2012.

\bibitem{sdca2013}
Shai Shalev-Shwartz and Tong Zhang.
\href{https://www.jmlr.org/papers/v14/shalev-shwartz13a.html}{Stochastic Dual Coordinate Ascent Methods for Regularized Loss Minimization}.
\emph{Journal of Machine Learning Research}, 14(16):567--599, 2013.

\bibitem{svrg2013}
Rie Johnson and Tong Zhang.
\href{https://papers.nips.cc/paper_files/paper/2013/hash/ac1dd209cbcc5e5d1c6e28598e8cbbe8-Abstract.html}{Accelerating
stochastic gradient descent using predictive variance reduction}.
In \emph{NeurIPS}, volume 26, 2013.

\bibitem{s2gd2013}
Jakub Konecny and Peter Richtarik.
\href{https://arxiv.org/abs/1312.1666}{Semi-Stochastic Gradient Descent Methods}.
\emph{arXiv preprint arXiv:1312.1666}, 2013.

\bibitem{saga2014}
Aaron Defazio, Francis Bach, and Simon Lacoste-Julien.
\href{https://arxiv.org/abs/1407.0202}{SAGA: A fast incremental gradient
method with support for non-strongly convex composite objectives}.
In \emph{NeurIPS}, volume 27, 2014.

\bibitem{finito2014}
Aaron J. Defazio, Tiberio S. Caetano, and Justin Domke.
\href{https://arxiv.org/abs/1407.2710}{Finito: A Faster, Permutable Incremental Gradient Method for Big Data Problems}.
\emph{arXiv preprint arXiv:1407.2710}, 2014.

\bibitem{proxsvrg2014}
Lin Xiao and Tong Zhang.
\href{https://arxiv.org/abs/1403.4699}{A proximal stochastic gradient
method with progressive variance reduction}.
\emph{SIAM Journal on Optimization}, 24(4):2057--2075, 2014.

\bibitem{catalyst2015}
Hongzhou Lin, Julien Mairal, and Zaid Harchaoui.
\href{https://papers.nips.cc/paper/2015/hash/c164bbc9d6c72a52c599bbb43d8db8e1-Abstract.html}{A Universal Catalyst for First-Order Optimization}.
In \emph{NeurIPS}, volume 28, 2015.

\bibitem{katyusha2016}
Zeyuan Allen-Zhu.
\href{https://arxiv.org/abs/1603.05953}{Katyusha: The first direct
acceleration of stochastic gradient methods}.
\emph{arXiv preprint arXiv:1603.05953}, 2016.

\bibitem{svrgpp2016}
Zeyuan Allen-Zhu and Yang Yuan.
\href{https://proceedings.mlr.press/v48/allen-zhub16.html}{Improved SVRG for Non-Strongly-Convex or Sum-of-Non-Convex Objectives}.
In \emph{ICML}, PMLR 48:1080--1089, 2016.

\bibitem{katyushax2018}
Zeyuan Allen-Zhu.
\href{https://proceedings.mlr.press/v80/allen-zhu18a.html}{Katyusha X: Simple Momentum Method for Stochastic Sum-of-Nonconvex Optimization}.
In \emph{ICML}, PMLR 80:179--185, 2018.

\bibitem{sagasampling2019}
Xun Qian, Zheng Qu, and Peter Richtarik.
\href{https://proceedings.mlr.press/v97/qian19a.html}{SAGA with Arbitrary Sampling}.
In \emph{ICML}, PMLR 97:5190--5199, 2019.

\bibitem{jacsketch2018}
Robert M. Gower, Peter Richtarik, and Francis Bach.
\href{https://arxiv.org/abs/1805.02632}{Stochastic Quasi-Gradient Methods: Variance Reduction via Jacobian Sketching}.
\emph{arXiv preprint arXiv:1805.02632}, 2018.

\bibitem{svrgsd2018}
Fanhua Shang, Yuanyuan Liu, Kaiwen Zhou, James Cheng, Kelvin Kai Wing Ng, and Yuichi Yoshida.
\href{https://proceedings.mlr.press/v84/shang18a.html}{Guaranteed Sufficient Decrease for Stochastic Variance Reduced Gradient Optimization}.
In \emph{AISTATS}, PMLR 84:1027--1036, 2018.

\bibitem{loopless2019}
Dmitry Kovalev, Samuel Horvath, and Peter Richtarik.
\href{https://arxiv.org/abs/1901.08689}{Don't Jump Through Hoops and Remove Those Loops: SVRG and Katyusha are Better Without the Outer Loop}.
\emph{arXiv preprint arXiv:1901.08689}, 2019.

\bibitem{looplesssampling2021}
Xun Qian, Zheng Qu, and Peter Richtarik.
\href{https://jmlr.org/papers/v22/20-156.html}{L-SVRG and L-Katyusha with Arbitrary Sampling}.
\emph{Journal of Machine Learning Research}, 22(112), 2021.

\bibitem{reddi2016}
Sashank J. Reddi, Ahmed Hefny, Suvrit Sra, Barnabas Poczos, and Alex Smola.
\href{https://proceedings.mlr.press/v48/reddi16.html}{Stochastic variance
reduction for nonconvex optimization}.
In \emph{International Conference on Machine Learning (ICML)},
PMLR 48:314--323, 2016.

\bibitem{allenzhu2016}
Zeyuan Allen-Zhu and Elad Hazan.
\href{https://proceedings.mlr.press/v48/allen-zhua16.html}{Variance reduction
for faster non-convex optimization}.
In \emph{ICML}, PMLR 48:699--707, 2016.

\bibitem{natasha2017}
Zeyuan Allen-Zhu.
\href{https://proceedings.mlr.press/v70/allen-zhu17a.html}{Natasha: Faster Non-Convex Stochastic Optimization via Strongly Non-Convex Parameter}.
In \emph{ICML}, PMLR 70:89--97, 2017.

\bibitem{sarah2017}
Lam M. Nguyen, Jie Liu, Katya Scheinberg, and Martin Tak\'{a}\v{c}.
\href{https://arxiv.org/abs/1703.00102}{SARAH: A novel method for machine
learning problems using stochastic recursive gradient}.
In \emph{ICML}, 2017.

\bibitem{isarah2018}
Lam M. Nguyen, Katya Scheinberg, and Martin Takac.
\href{https://arxiv.org/abs/1811.10105}{Inexact SARAH Algorithm for Stochastic Optimization}.
\emph{arXiv preprint arXiv:1811.10105}, 2018.

\bibitem{scsg2017}
Lihua Lei, Cheng Ju, Jianbo Chen, and Michael I. Jordan.
\href{https://arxiv.org/abs/1706.09156}{Non-convex finite-sum optimization
via SCSG methods}.
In \emph{NeurIPS}, volume 30, 2017.

\bibitem{spider2018}
Cong Fang, Chris Junchi Li, Zhouchen Lin, and Tong Zhang.
\href{https://arxiv.org/abs/1807.01695}{SPIDER: Near-optimal non-convex
optimization via stochastic path-integrated differential estimator}.
In \emph{NeurIPS}, volume 31, 2018.

\bibitem{snvrg2018}
Dongruo Zhou, Pan Xu, and Quanquan Gu.
\href{https://arxiv.org/abs/1806.07811}{Stochastic nested variance reduction
for nonconvex optimization}.
In \emph{NeurIPS}, volume 31, 2018.

\bibitem{spiderboost2019}
Zhe Wang, Kaiyi Ji, Yi Zhou, Yingbin Liang, and Vahid Tarokh.
\href{https://arxiv.org/abs/1810.10690}{SpiderBoost and momentum: Faster
stochastic variance reduction algorithms}.
In \emph{NeurIPS}, volume 32, 2019.

\bibitem{proxsarah2020}
Nhan H. Pham, Lam M. Nguyen, Dzung T. Phan, and Quoc Tran-Dinh.
\href{https://www.jmlr.org/papers/v21/19-248.html}{ProxSARAH: An Efficient Algorithmic Framework for Stochastic Composite Nonconvex Optimization}.
\emph{Journal of Machine Learning Research}, 21(110):1--48, 2020.

\bibitem{zerosarah2021}
Zhize Li, Slavomir Hanzely, and Peter Richtarik.
\href{https://arxiv.org/abs/2103.01447}{ZeroSARAH: Efficient Nonconvex Finite-Sum Optimization with Zero Full Gradient Computation}.
\emph{arXiv preprint arXiv:2103.01447}, 2021.

\bibitem{page2021}
Zhize Li, Hongyan Bao, Xiangliang Zhang, and Peter Richt\'{a}rik.
\href{https://proceedings.mlr.press/v139/li21a.html}{PAGE: A simple and
optimal probabilistic gradient estimator for nonconvex optimization}.
In \emph{ICML}, 2021.

\bibitem{nonconvexsampling2019}
Samuel Horvath and Peter Richtarik.
\href{https://proceedings.mlr.press/v97/horvath19a.html}{Nonconvex Variance Reduced Optimization with Arbitrary Sampling}.
In \emph{ICML}, PMLR 97:2781--2789, 2019.

\bibitem{storm2019}
Ashok Cutkosky and Francesco Orabona.
\href{https://arxiv.org/abs/1905.10018}{Momentum-based variance reduction
in non-convex SGD}.
In \emph{NeurIPS}, volume 32, 2019.

\bibitem{hybrid2019}
Quoc Tran-Dinh, Nhan H. Pham, Dzung T. Phan, and Lam M. Nguyen.
\href{https://arxiv.org/abs/1907.03793}{A hybrid stochastic optimization
framework for stochastic composite nonconvex optimization}.
\emph{arXiv preprint arXiv:1907.03793}, 2019.

\bibitem{stormplus2021}
Kfir Y. Levy, Ali Kavis, and Volkan Cevher.
\href{https://proceedings.neurips.cc/paper/2021/hash/ac10ff1941c540cd87c107330996f4f6-Abstract.html}{STORM+:
Fully adaptive SGD with recursive momentum for nonconvex optimization}.
In \emph{NeurIPS}, volume 34, 2021.

\bibitem{superadam2021}
Feihu Huang, Junyi Li, and Heng Huang.
\href{https://proceedings.neurips.cc/paper/2021/hash/4be5a36cbaca8ab9d2066debfe4e65c1-Abstract.html}{SUPER-ADAM: Faster and Universal Framework of Adaptive Gradients}.
In \emph{NeurIPS}, volume 34, 2021.

\bibitem{metastorm2022}
Zijian Liu, Ta Duy Nguyen, Thien Hang Nguyen, Alina Ene, and Huy L. Nguyen.
\href{https://arxiv.org/abs/2209.14853}{META-STORM: Generalized fully-adaptive
variance reduced SGD for unbounded functions}.
\emph{arXiv preprint arXiv:2209.14853}, 2022.

\bibitem{adaspider2022}
Ali Kavis, Stratis Skoulakis, Kimon Antonakopoulos, Leello Tadesse Dadi, and Volkan Cevher.
\href{https://proceedings.neurips.cc/paper_files/paper/2022/hash/94f625dcdec313cd432d65f96fcc51c8-Abstract-Conference.html}{Adaptive Stochastic Variance Reduction for Non-convex Finite-Sum Minimization}.
In \emph{NeurIPS}, volume 35, 2022.

\bibitem{jiang_adaptive2024}
Wei Jiang, Sifan Yang, Yibo Wang, and Lijun Zhang.
\href{https://proceedings.neurips.cc/paper_files/paper/2024/hash/272efd3a6091ceefcbc79f1f3a6fdba4-Abstract-Conference.html}{Adaptive
variance reduction for stochastic optimization under weaker assumptions}.
In \emph{NeurIPS}, volume 37, pages 22047--22080, 2024.

\bibitem{adasvrg2021}
Benjamin Dubois-Taine, Sharan Vaswani, Reza Babanezhad, Mark Schmidt, and Simon Lacoste-Julien.
\href{https://arxiv.org/abs/2102.09645}{SVRG Meets AdaGrad: Painless Variance Reduction}.
\emph{arXiv preprint arXiv:2102.09645}, 2021.

\bibitem{almosttune2020}
Bingcong Li, Lingda Wang, and Georgios B. Giannakis.
\href{https://proceedings.mlr.press/v119/li20n.html}{Almost Tune-Free Variance Reduction}.
In \emph{ICML}, PMLR 119:5969--5978, 2020.

\bibitem{adaptiveextra2022}
Zijian Liu, Ta Duy Nguyen, Alina Ene, and Huy L. Nguyen.
\href{https://proceedings.mlr.press/v162/liu22o.html}{Adaptive Accelerated (Extra-)Gradient Methods with Variance Reduction}.
In \emph{ICML}, PMLR 162:13947--13994, 2022.

\bibitem{alfpg2026}
Ganzhao Yuan.
\href{https://arxiv.org/abs/2502.21099v4}{Adaptive Lipschitz-Free Proximal Gradient Methods for Stochastic Composite Nonconvex Optimization}.
\emph{arXiv preprint arXiv:2502.21099v4}, 2026.

\bibitem{ssrgd2019}
Zhize Li.
\href{https://papers.nips.cc/paper_files/paper/2019/hash/addfa9b7e234254d26e9c7f2af1005cb-Abstract.html}{SSRGD: Simple Stochastic Recursive Gradient Descent for Escaping Saddle Points}.
In \emph{NeurIPS}, volume 32, 2019.

\bibitem{natasha22018}
Zeyuan Allen-Zhu.
\href{https://proceedings.neurips.cc/paper/2018/hash/79a49b3e3762632813f9e35f4ba53d6c-Abstract.html}{Natasha 2: Faster Non-Convex Optimization Than SGD}.
In \emph{NeurIPS}, volume 31, 2018.

\bibitem{svrc2019}
Zhe Wang, Yi Zhou, Yingbin Liang, and Guanghui Lan.
\href{https://proceedings.mlr.press/v89/wang19d.html}{Stochastic Variance-Reduced Cubic Regularization for Nonconvex Optimization}.
In \emph{AISTATS}, PMLR 89:2731--2740, 2019.

\bibitem{svrcnewton2018}
Dongruo Zhou, Pan Xu, and Quanquan Gu.
\href{https://proceedings.mlr.press/v80/zhou18d.html}{Stochastic Variance-Reduced Cubic Regularized Newton Methods}.
In \emph{ICML}, PMLR 80:5990--5999, 2018.

\bibitem{lower2023}
Yossi Arjevani, Yair Carmon, John C. Duchi, Dylan J. Foster,
Nathan Srebro, and Blake Woodworth.
\href{https://doi.org/10.1007/s10107-022-01822-7}{Lower bounds for non-convex
stochastic optimization}.
\emph{Mathematical Programming}, 199:165--214, 2023.

\bibitem{scgd2014}
Mengdi Wang, Ethan X. Fang, and Han Liu.
\href{https://arxiv.org/abs/1411.3803}{Stochastic compositional gradient
descent: Algorithms for minimizing compositions of expected-value functions}.
\emph{arXiv preprint arXiv:1411.3803}, 2014.

\bibitem{asc2017}
Mengdi Wang, Ji Liu, and Ethan X. Fang.
\href{https://jmlr.org/papers/v18/16-504.html}{Accelerating stochastic
composition optimization}.
\emph{Journal of Machine Learning Research}, 18(105):1--23, 2017.

\bibitem{nasa2020}
Saeed Ghadimi, Andrzej Ruszczynski, and Mengdi Wang.
\href{https://doi.org/10.1137/18M1230542}{A Single Timescale Stochastic Approximation Method for Nested Stochastic Optimization}.
\emph{SIAM Journal on Optimization}, 2020.

\bibitem{nearlyeasy2020}
Tianyi Chen, Yuejiao Sun, and Wotao Yin.
\href{https://arxiv.org/abs/2008.10847}{Solving Stochastic Compositional Optimization is Nearly as Easy as Solving Stochastic Optimization}.
\emph{arXiv preprint arXiv:2008.10847}, 2020.

\bibitem{lian2017}
Xiangru Lian, Mengdi Wang, and Ji Liu.
\href{https://proceedings.mlr.press/v54/lian17a.html}{Finite-sum composition
optimization via variance reduced gradient descent}.
In \emph{International Conference on Artificial Intelligence and Statistics
(AISTATS)}, PMLR 54:1159--1167, 2017.

\bibitem{compavg2018}
Tsung-Yu Hsieh, Yasser El-Manzalawy, Yiwei Sun, and Vasant Honavar.
\href{https://arxiv.org/abs/1809.01225}{Compositional Stochastic Average Gradient for Machine Learning and Related Applications}.
\emph{arXiv preprint arXiv:1809.01225}, 2018.

\bibitem{civr2019}
Junyu Zhang and Lin Xiao.
\href{https://proceedings.mlr.press/v97/zhang19n.html}{A composite randomized
incremental gradient method}.
In \emph{ICML}, PMLR 97:7454--7462, 2019.

\bibitem{sarahcomp2019}
Wenqing Hu, Chris Junchi Li, Xiangru Lian, Ji Liu, and Huizhuo Yuan.
\href{https://papers.neurips.cc/paper_files/paper/2019/hash/21ce689121e39821d07d04faab328370-Abstract.html}{Efficient Smooth Non-Convex Stochastic Compositional Optimization via Stochastic Recursive Gradient Descent}.
In \emph{NeurIPS}, volume 32, 2019.

\bibitem{nested2019}
Junyu Zhang and Lin Xiao.
\href{https://arxiv.org/abs/1908.11468}{Multi-level composite stochastic
optimization via nested variance reduction}.
\emph{arXiv preprint arXiv:1908.11468}, 2019.

\bibitem{jiang_multilevel2022}
Wei Jiang, Bokun Wang, Yibo Wang, Lijun Zhang, and Tianbao Yang.
\href{https://proceedings.mlr.press/v162/jiang22c.html}{Optimal algorithms
for stochastic multi-level compositional optimization}.
In \emph{ICML}, PMLR 162:10195--10216, 2022.

\bibitem{levelindependent2020}
Krishnakumar Balasubramanian, Saeed Ghadimi, and Anthony Nguyen.
\href{https://arxiv.org/abs/2008.10526}{Stochastic Multi-level Composition Optimization Algorithms with Level-Independent Convergence Rates}.
\emph{arXiv preprint arXiv:2008.10526}, 2020.

\bibitem{jiang_revisiting2025}
Wei Jiang, Sifan Yang, Yibo Wang, Tianbao Yang, and Lijun Zhang.
\href{https://doi.org/10.1109/TPAMI.2025.3552197}{Revisiting stochastic
multi-level compositional optimization}.
\emph{IEEE Transactions on Pattern Analysis and Machine Intelligence},
47(7):5613--5624, 2025.

\bibitem{fccotheory2022}
Bokun Wang and Tianbao Yang.
\href{https://proceedings.mlr.press/v162/wang22ak.html}{Finite-Sum Coupled Compositional Stochastic Optimization: Theory and Applications}.
In \emph{ICML}, PMLR 162:23292--23317, 2022.

\bibitem{jiang_msvr2022}
Wei Jiang, Gang Li, Yibo Wang, Lijun Zhang, and Tianbao Yang.
\href{https://papers.nips.cc/paper_files/paper/2022/hash/d13ee73683fd5567e5c07634a25cd7b8-Abstract-Conference.html}{Multi-block-single-probe
variance reduced estimator for coupled compositional optimization}.
In \emph{NeurIPS}, volume 35, pages 32499--32511, 2022.

\bibitem{weakcomp2023}
Quanqi Hu, Dixian Zhu, and Tianbao Yang.
\href{https://papers.nips.cc/paper_files/paper/2023/hash/1160792eab11de2bbaf9e71fce191e8c-Abstract-Conference.html}{Non-Smooth Weakly-Convex Finite-sum Coupled Compositional Optimization}.
In \emph{NeurIPS}, volume 36, 2023.

\bibitem{jiang_fcco2026}
Wei Jiang, Sifan Yang, Yibo Wang, Lijun Zhang, and Zechao Li.
\href{https://arxiv.org/abs/2609.15723}{Solving finite-sum coupled
compositional optimization via multi-block-single-probe estimator}.
\emph{arXiv preprint arXiv:2609.15723}, 2026.

\bibitem{momentumcomp2025}
Xingyu Chen, Bokun Wang, Ming Yang, Qihang Lin, and Tianbao Yang.
\href{https://papers.nips.cc/paper_files/paper/2025/hash/714467868756e0e0614e99de02ac0032-Abstract-Conference.html}{Stochastic Momentum Methods for Non-smooth Non-Convex Finite-Sum Coupled Compositional Optimization}.
In \emph{NeurIPS}, volume 38, 2025.

\bibitem{vrfw2016}
Elad Hazan and Haipeng Luo.
\href{https://proceedings.mlr.press/v48/hazana16.html}{Variance-Reduced and Projection-Free Stochastic Optimization}.
In \emph{ICML}, PMLR 48:1263--1271, 2016.

\bibitem{projectioncomp2022}
Tesi Xiao, Krishnakumar Balasubramanian, and Saeed Ghadimi.
\href{https://proceedings.neurips.cc/paper_files/paper/2022/hash/7e16384b94a1c7e4462a70bb8fb93ca9-Abstract-Conference.html}{A Projection-free Algorithm for Constrained Stochastic Multi-level Composition Optimization}.
In \emph{NeurIPS}, volume 35, 2022.

\bibitem{jiang_projection2024}
Wei Jiang, Sifan Yang, Wenhao Yang, Yibo Wang, Yuanyu Wan, and Lijun Zhang.
\href{https://proceedings.mlr.press/v235/jiang24g.html}{Projection-free
variance reduction methods for stochastic constrained multi-level
compositional optimization}.
In \emph{ICML}, PMLR 235:21962--21987, 2024.

\bibitem{scgpp2019}
Hamed Hassani, Amin Karbasi, Aryan Mokhtari, and Zebang Shen.
\href{https://arxiv.org/abs/1902.06992}{Stochastic Conditional Gradient++}.
\emph{arXiv preprint arXiv:1902.06992}, 2019.

\bibitem{jiang_projection2026}
Wei Jiang, Sifan Yang, Wenhao Yang, Yibo Wang, Yuanyu Wan,
Zechao Li, and Lijun Zhang.
\href{https://arxiv.org/abs/2609.15679}{Projection-free multi-level
algorithms for stochastic constrained compositional optimization}.
\emph{arXiv preprint arXiv:2609.15679}, 2026.

\bibitem{jiang_models2023}
Wei Jiang, Jiayu Qin, Lingyu Wu, Changyou Chen, Tianbao Yang, and Lijun Zhang.
\href{https://proceedings.mlr.press/v202/jiang23g.html}{Learning unnormalized
statistical models via compositional optimization}.
In \emph{ICML}, PMLR 202:15105--15124, 2023.

\bibitem{fednest2022}
Davoud Ataee Tarzanagh, Mingchen Li, Christos Thrampoulidis, and Samet Oymak.
\href{https://proceedings.mlr.press/v162/tarzanagh22a.html}{FedNest: Federated Bilevel, Minimax, and Compositional Optimization}.
In \emph{ICML}, PMLR 162:21146--21179, 2022.

\bibitem{jiang_models2026}
Wei Jiang, Jiayu Qin, Lingyu Wu, Changyou Chen, Tianbao Yang, and Lijun Zhang.
\href{https://jiangw1998.github.io/pdf/TPAMI-2026-Jiang.pdf}{Optimizing
unnormalized statistical models through compositional optimization}.
\emph{IEEE Transactions on Pattern Analysis and Machine Intelligence},
48(2):1949--1960, 2026.

\bibitem{localcomp2022}
Hongchang Gao, Junyi Li, and Heng Huang.
\href{https://proceedings.mlr.press/v162/gao22c.html}{On the Convergence of Local Stochastic Compositional Gradient Descent with Momentum}.
In \emph{ICML}, PMLR 162:7017--7035, 2022.

\bibitem{qsgd2017}
Dan Alistarh, Demjan Grubic, Jerry Li, Ryota Tomioka, and Milan Vojnovic.
\href{https://papers.nips.cc/paper/2017/hash/6c340f25839e6acdc73414517203f5f0-Abstract.html}{QSGD: Communication-Efficient SGD via Gradient Quantization and Encoding}.
In \emph{NeurIPS}, volume 30, 2017.

\bibitem{terngrad2017}
Wei Wen, Cong Xu, Feng Yan, Chunpeng Wu, Yandan Wang, Yiran Chen, and Hai Li.
\href{https://arxiv.org/abs/1705.07878}{TernGrad: Ternary Gradients to Reduce Communication in Distributed Deep Learning}.
In \emph{NeurIPS}, volume 30, 2017.

\bibitem{signsgd2018}
Jeremy Bernstein, Yu-Xiang Wang, Kamyar Azizzadenesheli, and Anima Anandkumar.
\href{https://proceedings.mlr.press/v80/bernstein18a.html}{signSGD:
Compressed optimisation for non-convex problems}.
In \emph{ICML}, PMLR 80:560--569, 2018.

\bibitem{majority2019}
Jeremy Bernstein, Jiawei Zhao, Kamyar Azizzadenesheli, and Anima Anandkumar.
\href{https://arxiv.org/abs/1810.05291}{signSGD with majority vote is
communication efficient and fault tolerant}.
In \emph{International Conference on Learning Representations (ICLR)}, 2019.

\bibitem{errorfeedback2019}
Sai Praneeth Karimireddy, Quentin Rebjock, Sebastian U. Stich, and Martin Jaggi.
\href{https://proceedings.mlr.press/v97/karimireddy19a.html}{Error feedback
fixes SignSGD and other gradient compression schemes}.
In \emph{ICML}, PMLR 97:3252--3261, 2019.

\bibitem{jiang_sign2024}
Wei Jiang, Sifan Yang, Wenhao Yang, and Lijun Zhang.
\href{https://papers.nips.cc/paper_files/paper/2024/hash/3b7db05c0e383518d789b6e93131f1f0-Abstract-Conference.html}{Efficient
sign-based optimization: Accelerating convergence via variance reduction}.
In \emph{NeurIPS}, volume 37, pages 33891--33932, 2024.

\bibitem{lion2023}
Xiangning Chen, Chen Liang, Da Huang, Esteban Real, Kaiyuan Wang,
Hieu Pham, Xuanyi Dong, Thang Luong, Cho-Jui Hsieh, Yifeng Lu, and Quoc V. Le.
\href{https://proceedings.neurips.cc/paper_files/paper/2023/hash/9a39b4925e35cf447ccba8757137d84f-Abstract.html}{Symbolic
discovery of optimization algorithms}.
In \emph{NeurIPS}, volume 36, 2023.

\bibitem{jiang_lion2026}
Wei Jiang, Mao Xu, Wenhao Yang, Yibo Wang, Zechao Li, and Lijun Zhang.
\href{https://proceedings.mlr.press/v306/jiang26ax.html}{Convergence analysis
of the Lion optimizer in centralized and distributed settings}.
In \emph{ICML}, PMLR 306:53308--53333, 2026.

\bibitem{diana2019}
Konstantin Mishchenko, Eduard Gorbunov, Martin Takac, and Peter Richtarik.
\href{https://arxiv.org/abs/1901.09269}{Distributed Learning with Compressed Gradient Differences}.
\emph{arXiv preprint arXiv:1901.09269}, 2019.

\bibitem{marina2021}
Eduard Gorbunov, Konstantin P. Burlachenko, Zhize Li, and Peter Richtarik.
\href{https://proceedings.mlr.press/v139/gorbunov21a.html}{MARINA: Faster Non-Convex Distributed Learning with Compression}.
In \emph{ICML}, PMLR 139:3788--3798, 2021.

\bibitem{jiang_distributed2026}
Wei Jiang, Zechao Li, and Lijun Zhang.
\href{https://arxiv.org/abs/2609.18656}{Revisiting distributed sign-based
variance reduction}.
\emph{arXiv preprint arXiv:2609.18656}, 2026.

\bibitem{ef212021}
Peter Richtarik, Igor Sokolov, and Ilyas Fatkhullin.
\href{https://papers.nips.cc/paper_files/paper/2021/hash/231141b34c82aa95e48810a9d1b33a79-Abstract.html}{EF21: A New, Simpler, Theoretically Better, and Practically Faster Error Feedback}.
In \emph{NeurIPS}, volume 34, 2021.

\bibitem{dasha2022}
Alexander Tyurin and Peter Richtarik.
\href{https://arxiv.org/abs/2202.01268}{DASHA: Distributed Nonconvex Optimization with Communication Compression, Optimal Oracle Complexity, and No Client Synchronization}.
\emph{arXiv preprint arXiv:2202.01268}, 2022.

\bibitem{scaffold2020}
Sai Praneeth Karimireddy, Satyen Kale, Mehryar Mohri, Sashank Reddi, Sebastian Stich, and Ananda Theertha Suresh.
\href{https://proceedings.mlr.press/v119/karimireddy20a.html}{SCAFFOLD: Stochastic Controlled Averaging for Federated Learning}.
In \emph{ICML}, PMLR 119:5132--5143, 2020.

\bibitem{fedpage2021}
Haoyu Zhao, Zhize Li, and Peter Richtarik.
\href{https://arxiv.org/abs/2108.04755}{FedPAGE: A Fast Local Stochastic Gradient Method for Communication-Efficient Federated Learning}.
\emph{arXiv preprint arXiv:2108.04755}, 2021.

\bibitem{hybriddecentral2021}
Ran Xin, Usman A. Khan, and Soummya Kar.
\href{https://proceedings.mlr.press/v139/xin21a.html}{A Hybrid Variance-Reduced Method for Decentralized Stochastic Non-Convex Optimization}.
In \emph{ICML}, PMLR 139:11459--11469, 2021.

\bibitem{proxskip2022}
Konstantin Mishchenko, Grigory Malinovsky, Sebastian Stich, and Peter Richtarik.
\href{https://proceedings.mlr.press/v162/mishchenko22b.html}{ProxSkip: Yes! Local Gradient Steps Provably Lead to Communication Acceleration! Finally!}.
In \emph{ICML}, PMLR 162:15750--15769, 2022.

\bibitem{saddle2016}
P. Balamurugan and Francis Bach.
\href{https://arxiv.org/abs/1605.06398}{Stochastic variance reduction
methods for saddle-point problems}.
In \emph{NeurIPS}, volume 29, 2016.

\bibitem{matrixgames2019}
Yair Carmon, Yujia Jin, Aaron Sidford, and Kevin Tian.
\href{https://papers.nips.cc/paper_files/paper/2019/hash/6c442e0e996fa84f344a14927703a8c1-Abstract.html}{Variance Reduction for Matrix Games}.
In \emph{NeurIPS}, volume 32, 2019.

\bibitem{svre2019}
Tatjana Chavdarova, Gauthier Gidel, Francois Fleuret, and Simon Lacoste-Julien.
\href{https://papers.nips.cc/paper_files/paper/2019/hash/58a2fc6ed39fd083f55d4182bf88826d-Abstract.html}{Reducing Noise in GAN Training with Variance Reduced Extragradient}.
In \emph{NeurIPS}, volume 32, 2019.

\bibitem{catalystminimax2020}
Junchi Yang, Siqi Zhang, Negar Kiyavash, and Niao He.
\href{https://papers.nips.cc/paper_files/paper/2020/hash/3db54f5573cd617a0112d35dd1e6b1ef-Abstract.html}{A Catalyst Framework for Minimax Optimization}.
In \emph{NeurIPS}, volume 33, 2020.

\bibitem{sreda2020}
Luo Luo, Haishan Ye, Zhichao Huang, and Tong Zhang.
\href{https://papers.nips.cc/paper/2020/hash/ecb47fbb07a752413640f82a945530f8-Abstract.html}{Stochastic
recursive gradient descent ascent for stochastic nonconvex-strongly-concave
minimax problems}.
In \emph{NeurIPS}, volume 33, 2020.

\bibitem{hybridminimax2020}
Quoc Tran-Dinh, Deyi Liu, and Lam M. Nguyen.
\href{https://papers.nips.cc/paper/2020/hash/7f141cf8e7136ce8701dc6636c2a6fe4-Abstract.html}{Hybrid Variance-Reduced SGD Algorithms for Minimax Problems with Nonconvex-Linear Function}.
In \emph{NeurIPS}, volume 33, 2020.

\bibitem{adgda2023}
Feihu Huang, Xidong Wu, and Zhengmian Hu.
\href{https://proceedings.mlr.press/v206/huang23a.html}{AdaGDA: Faster
adaptive gradient descent ascent methods for minimax optimization}.
In \emph{AISTATS}, PMLR 206:2365--2389, 2023.

\bibitem{dream2024}
Lesi Chen, Haishan Ye, and Luo Luo.
\href{https://proceedings.mlr.press/v238/chen24b.html}{An Efficient Stochastic Algorithm for Decentralized Nonconvex-Strongly-Concave Minimax Optimization}.
In \emph{AISTATS}, PMLR 238:1990--1998, 2024.

\bibitem{vi2022}
Ahmet Alacaoglu and Yura Malitsky.
\href{https://proceedings.mlr.press/v178/alacaoglu22a.html}{Stochastic
variance reduction for variational inequality methods}.
In \emph{Conference on Learning Theory (COLT)}, PMLR 178:778--816, 2022.

\end{thebibliography}
\end{document}